\pdfoutput=1
\RequirePackage{ifpdf}
\ifpdf 
\documentclass[pdftex]{sigma}
\else
\documentclass{sigma}
\fi

\numberwithin{equation}{section}

\newtheorem{Theorem}{Theorem}[section]
\newtheorem*{Theorem*}{Theorem}

\newtheorem{prop}[Theorem]{Proposition}
 { \theoremstyle{definition}

 }

\newcommand{\Hom}{\operatorname{Hom}}
\newcommand{\End}{\operatorname{End}}
\DeclareSymbolFont{script}{U}{eus}{m}{n}
\DeclareMathSymbol{\Wedge}{0}{script}{"5E}
\newcommand{\bthree}[3]{\begin{picture}(48,11)
\put(4,1.5){\makebox(0,0){$\times$}}
\put(24,1.5){\makebox(0,0){$\bullet$}}
\put(44,1.5){\makebox(0,0){$\bullet$}}
\put(4,1.5){\line(1,0){20}}
\put(23,2.5){\line(1,0){20}}
\put(23,0.5){\line(1,0){20}}
\put(34,1.5){\makebox(0,0){$\rangle$}}
\put(4,10){\makebox(0,0){$\vphantom{(}\scriptstyle #1$}}
\put(24,10){\makebox(0,0){$\vphantom{(}\scriptstyle #2$}}
\put(44,10){\makebox(0,0){$\vphantom{(}\scriptstyle #3$}}
\end{picture}}

\newcommand{\xox}[3]{\begin{picture}(40,11)
\put(4,1.5){\makebox(0,0){$\times$}}
\put(20,1.5){\makebox(0,0){$\bullet$}}
\put(36,1.5){\makebox(0,0){$\times$}}
\put(4,1.5){\line(1,0){32}}
\put(4,8){\makebox(0,0){$\scriptstyle #1$}}
\put(20,8){\makebox(0,0){$\scriptstyle #2$}}
\put(36,8){\makebox(0,0){$\scriptstyle #3$}}
\end{picture}}
\newcommand{\btwo}[2]{\begin{picture}(28,11)
\put(4,1.5){\makebox(0,0){$\bullet$}}
\put(24,1.5){\makebox(0,0){$\times$}}
\put(23,2.5){\line(-1,0){20}}
\put(23,0.5){\line(-1,0){20}}
\put(14,1.5){\makebox(0,0){$\langle$}}
\put(4,8){\makebox(0,0){$\scriptstyle #1$}}
\put(24,8){\makebox(0,0){$\scriptstyle #2$}}
\end{picture}}
\newcommand{\gtwo}[2]{\begin{picture}(28,10)
\put(4,1.1){\makebox(0,0){$\bullet$}}
\put(24,1.5){\makebox(0,0){$\times$}}
\put(4,3.5){\line(1,0){18}}
\put(4,1.5){\line(1,0){20}}
\put(4,-.5){\line(1,0){18}}
\put(14,1.5){\makebox(0,0){$\langle$}}
\put(4,8){\makebox(0,0){$\scriptstyle #1$}}
\put(24,8){\makebox(0,0){$\scriptstyle #2$}}
\end{picture}}

\begin{document}
\allowdisplaybreaks

\newcommand{\arXivNumber}{2201.13048}

\renewcommand{\thefootnote}{}

\renewcommand{\PaperNumber}{031}

\FirstPageHeading

\ShortArticleName{Spinors in Five-Dimensional Contact Geometry}

\ArticleName{Spinors in Five-Dimensional Contact Geometry\footnote{This paper is a~contribution to the Special Issue on Twistors from Geometry to Physics in honour of Roger Penrose. The~full collection is available at \href{https://www.emis.de/journals/SIGMA/Penrose.html}{https://www.emis.de/journals/SIGMA/Penrose.html}}}

\Author{Michael EASTWOOD~$^{\rm a}$ and Timothy MOY~$^{\rm b}$}

\AuthorNameForHeading{M.~Eastwood and T.~Moy}

\Address{$^{\rm a)}$~School of Mathematical Sciences, University of Adelaide, SA 5005, Australia}
\EmailD{\href{mailto:meastwoo@gmail.com}{meastwoo@gmail.com}}

\Address{$^{\rm b)}$~Clare College, University of Cambridge, CB2 1TL, England, UK}
\EmailD{\href{mailto:tjahm2@cam.ac.uk}{tjahm2@cam.ac.uk}}

\ArticleDates{Received January 31, 2022, in final form April 13, 2022; Published online April 16, 2022}

\Abstract{We use classical (Penrose) two-component spinors to set up the differential geometry of two parabolic contact structures in five dimensions, namely $G_2$ contact geometry and Legendrean contact geometry. The key players in these two geometries are invariantly defined directional derivatives defined only in the contact directions. We explain how to define them and their usage in constructing basic invariants such as the harmonic curvature, the obstruction to being locally flat from the parabolic viewpoint. As an application, we calculate the invariant torsion of the $G_2$ contact structure on the configuration space of a~flying saucer (always a five-dimensional contact manifold).}


\Classification{53B05; 53D10; 58J10}

\begin{flushright}
\begin{minipage}{60mm}
\it Dedicated to Roger Penrose\\
 on the occasion of his 90th birthday
\end{minipage}
\end{flushright}

\renewcommand{\thefootnote}{\arabic{footnote}}
\setcounter{footnote}{0}

\section{Introduction}
Two-component spinors are widely used in four-dimensional Lorentzian geometry.
The seminal books `Spinors and space-time'~\cite{OT,NT} are devoted to such
usage. At a very basic level, two-component spinors arise via the 2--1 covering
of Lie groups
\[{\mathrm{SL}}(2,{\mathbb{C}})\longrightarrow{\mathrm{SO}}^\uparrow(3,1),\]
where ${\mathrm{SO}}^\uparrow(3,1)$ is the identity-connected component of the
Lorentz group. Similarly, in three dimensions, the 2--1 covering
\[{\mathrm{SL}}(2,{\mathbb{R}})\longrightarrow{\mathrm{SO}}^\uparrow(2,1)\]
is responsible for the utility of two-component spinors in three dimensions
(as in Section~\ref{3d_conformal}).

At a very basic level, the two-component spinors in this article arise via
inclusions
\[{\mathrm{SL}}(2,{\mathbb{R}})\hookrightarrow{\mathrm{Sp}}(4,{\mathbb{R}}),\]
where ${\mathrm{Sp}}(4,{\mathbb{R}})$ is the subgroup of
${\mathrm{SL}}(4,{\mathbb{R}})$ consisting of matrices that preserve a fixed
non\-degenerate skew form on ${\mathbb{R}}^4$. There are two such inclusions:
\begin{itemize}\itemsep=0pt
\item by writing ${\mathbb{R}}^4=\bigodot^3{\mathbb{R}}^2$, where $\bigodot$
denotes symmetric tensor product,
\item by writing ${\mathbb{R}}^4={\mathbb{R}}^2\oplus{\mathbb{R}}^2$,
\end{itemize}
giving rise to $G_2$ contact geometry (as in Section~\ref{G2}) and Legendrean
contact geometry (as in Section~\ref{CL}), respectively. These two geometries are
defined on a five-dimensional contact mani\-fold~$M$ as extra structure on the
contact distribution~$H\subset TM$. A contact form $\alpha$ is a $1$-form so
that $H=\ker\alpha$. It gives rise to a nondegenerate skew form, the Levi
form, namely ${\rm d}\alpha|_H$. The extra spin structures in Sections~\ref{G2} and~\ref{CL} are required to be compatible with the Levi form.

Another important theme in \cite{OT,NT} is conformal geometry. It is concerned
with what happens if the (Lorentzian) metric is replaced by a smooth positive
multiple of itself. The resulting formul{\ae}, for
example~(\ref{change_of_spin_connection}), fit very well with spinors. For the
contact geometries in this article, the corresponding freedom is in choosing a
contact form~$\alpha$. The resulting formul{\ae}, for
example~(\ref{change_of_G2_spin_connection}), (\ref{full_transformation_law}),
(\ref{full_barred_transformation_law}), also fit very well with spinors. These
`conformal' structures are examples of {\em parabolic geometries}~\cite{CS}.
In particular, {\em invariant differential operators} play a key part in
parabolic constructions.

Ideally, one would like to approach the natural differential geometric calculus
on these various geometries via invariant differential operators. More
specifically, for a chosen `scale' (a metric in the conformal class or a choice
of contact form) one expects a canonical (partial) connection on all the
natural irreducible vector bundles. This expectation follows from the
\v{C}ap-Slov\'ak theory of {\em Weyl structures and scales} in parabolic
geometry~(see \cite{CSweyl} or \cite[Section~5.1]{CS}). In conformal geometry, it is
just the Levi-Civita connection: see Sections~\ref{LC} and~\ref{conf} for details.
For parabolic contact structures, one expects partial connections defined only
in the contact directions. We shall see in Sections~\ref{G2} and~\ref{CL} that a
suitable collection of invariant operators is, indeed, sufficient for these
purposes. For the five-dimensional Legendrean contact structures
in Section~\ref{CL}, these operators can easily be found (within the Rumin complex,
explained in Section~\ref{rumin}). The two key invariant operators~(\ref{basic_invariant_operators}) in~$G_2$ contact geometry remain somewhat
mysterious: it is shown in Section~\ref{G2} that these operators are, indeed,
invariant but we have not been able to find a~suitable shortcut to their
construction. Nevertheless, we are able to construct, in Section~\ref{saucers}, the
general~$G_2$ contact geometry from a suitable Legendrean contact structure
and, thereby, calculate its basic spinor invariant (a certain septic).

Through this article we shall use Penrose's abstract index notation and other
conventions from~\cite{OT}. Also, to ease the notational burden, we shall not
carefully distinguish between bundles and sections thereof, for example writing
$\omega_a\in\Wedge^1$ instead of $\omega_a\in\Gamma\big({\Wedge}^1\big)$ or even
$\omega_a\in\Gamma\big(M,\Wedge^1\big)$, for a $1$-form~$\omega_a$. Especially as this
article is concerned with {\em local} differential geometry, this should
cause no confusion.

\section{The Levi-Civita connection}\label{LC}
On a general smooth manifold, the exterior derivative and the Lie derivative
are defined independently of local co\"ordinates (and there is little else with
this property~\cite{KMS}). Both of these operations can be defined in terms of
an arbitrary torsion-free affine connection. For the exterior derivative on
$1$-forms, we have (following the conventions of~\cite[equation~(4.3.14)]{OT})
\[\omega_b\longmapsto\nabla_{[a}\omega_{b]}.\]
For the Lie derivative on covariant $2$-tensors, we have
\begin{equation}\label{Lie}
\phi_{bc}\longmapsto({\mathcal{L}}_X\phi)_{bc}\equiv
X^a\nabla_a\phi_{bc}+(\nabla_bX^a)\phi_{ac}+(\nabla_cX^a)\phi_{ba},
\end{equation}
for a vector field~$X^a$ and, irrespective of the usual interpretation of
${\mathcal{L}}_X$ in terms of the flow of~$X^a$, it easy to check that this
expression does not depend on the choice of~$\nabla_a$. It is convenient to
regard the right hand side of~(\ref{Lie}), let's say for a symmetric
tensor~$\phi_{ab}$, as an invariantly defined differential {\em pairing}
\[\textstyle TM\times\bigodot^2\!\Wedge^1\longrightarrow
\bigodot^2\!\Wedge^1.\]
In particular, if $g_{ab}$ is a semi-Riemannian metric, that is to say a
nondegenerate symmetric form, then the Lie derivative of~$g_{ab}$, with a
convenient factor of~$\frac12$ thrown in, can be regarded as a~canonically
defined linear differential operator
\begin{equation}\label{un}
\textstyle TM\ni
X^a\longmapsto
\frac12{\mathcal{L}}_Xg_{bc}\in\bigodot^2\!\Wedge^1.\end{equation}
Of course,
the tensor $g_{ab}$ also defines an isomorphism $TM\cong\Wedge^1$ by
$X^a\mapsto g_{ab}X^b$, and so we have obtained an invariantly defined linear
differential operator
\begin{equation}\label{deux}
\textstyle\Wedge^1\cong TM\longrightarrow\bigodot^2\!\Wedge^1.\end{equation}
In combination with the exterior derivative ${\rm d}\colon \Wedge^1\to\Wedge^2$, we have
obtained
\begin{equation}\label{trois}
\textstyle\Wedge^1\longrightarrow\bigodot^2\!\Wedge^1\oplus\Wedge^2=
\Wedge^1\otimes\Wedge^1\end{equation}
and we claim that this is the Levi-Civita connection defined by~$g_{ab}$. This
is easy to check: if we use the Levi-Civita connection in~(\ref{Lie}), then
\begin{align*}\tfrac12{\mathcal{L}}_Xg_{ab}+\nabla_{[a}X_{b]}
&= \tfrac12\big(X^c\nabla_cg_{ab}+(\nabla_aX^c)g_{cb}+(\nabla_bX^c)g_{ac}\big) +\nabla_{[a}X_{b]}\\
&= \nabla_{(a}X_{b)}+\nabla_{[a}X_{b]}=\nabla_aX_b,
\end{align*}
as required. But we can play the moves (\ref{un}), (\ref{deux}), and
(\ref{trois}) to {\em define} the Levi-Civita connection. An advantage of
this viewpoint is that it is easily modified to define other connections (and
partial connections), as we shall soon see.

\section{Conformal differential geometry}\label{conf}
In semi-Riemannian conformal differential geometry, instead of a
metric~$g_{ab}$, we are given only a~{\em conformal class} of metrics. A
convenient way of expressing this is to say that we are given~$\eta_{ab}$, a
nondegenerate section of $\bigodot^2\!\Wedge^1\otimes L^2$ for some line
bundle~$L$. Thus, if $\sigma$ is a non-vanishing section of~$L$, then
$g_{ab}\equiv\sigma^{-2}\eta_{ab}$ is a genuine metric. We shall refer to
$\sigma$ as a {\em scale} and, if we choose a different section
$\widehat\sigma=\Omega^{-1}\sigma$ for some nowhere vanishing
function~$\Omega$, then we encounter a {\em rescaled} metric
$\widehat{g}_{ab}\equiv\widehat\sigma^{-2}\eta_{ab}=\Omega^2g_{ab}$. For a~given conformal structure~$\eta_{ab}$, let us define~$\eta^{ab}$, a~section of
$\bigodot^2\!TM\otimes L^{-2}$, by $\eta_{ac}\eta^{bc}=\delta_a{}^b$, where
$\delta_a{}^b$ is the invariant (Kronecker delta) pairing
$TM\otimes\Wedge^1\to\Wedge^0$.

On an $n$-dimensional oriented conformal manifold, we may attempt to normalise
a section $\epsilon_{ab\dots d}$ of the bundle $\Wedge^n$, compatible
with the orientation, by insisting that
\begin{equation}\label{normalise}
\epsilon_{ab\dots d}\epsilon_{ef\dots h}\eta^{ae}\eta^{bf}\cdots\eta^{dh}
\equiv n! ,\end{equation}
just as we may do on an $n$-dimensional oriented Riemannian manifold to
normalise and define the volume form. The problem with (\ref{normalise}) is
that the left hand side takes values in the line bundle~$L^{-2n}$. To remedy
this problem, we may insist that~$\epsilon_{ab\dots d}$ take values in
$\Wedge^n\otimes L^n$. The normalisation (\ref{normalise}) now makes good
sense and $\epsilon_{ab\dots d}$ is uniquely determined. Equivalently, we may
view $\epsilon_{ab\dots d}$ as providing a canonical identification
\[L^{-n}=\Wedge^n\qquad
\mbox{given by}\quad\phi\mapsto\phi\epsilon_{ab\dots d}.\]
It it usual to write $L\equiv\Wedge^0[1]$ and refer to sections of
$L^w\equiv\Wedge^0[w]$ as {\em conformal densities of weight}~$w$. In the
presence of a scale $\sigma$, defining a metric $g_{ab}=\sigma^{-2}\eta_{ab}$,
the normalisation (\ref{normalise}) reads
\[(\sigma^n\epsilon_{ab\dots d})(\sigma^n\epsilon_{ef\dots h})
g^{ae}g^{bf}\cdots g^{dh}=n!\]
so $\sigma^n\epsilon_{ab\dots d}$ is the usual volume form for the
metric~$g_{ab}$. This is consistent with the scale $\sigma$ trivialising
all the density bundles $\Wedge^0[w]$. To summarise, in the presence of a
scale~$\sigma$, a conformal density of weight $w$ can be regarded as an ordinary
function but if we change to a new scale $\widehat\sigma=\Omega^{-1}\sigma$,
then the same density is represented by a new function $\widehat{f}=\Omega^wf$.

Once we have decided that $\Wedge^0[1]=L=(\Wedge^n)^{-1/n}$, the Lie
derivative ${\mathcal{L}}_X\colon \Wedge^1\to\Wedge^1$ induces invariantly
defined differential pairings $TM\times B\to B$ for all conformally weighted
tensor bundles~$B$ and, in particular,
\[\textstyle TM\times\bigodot^2\!\Wedge^1[2]\to\bigodot^2\!\Wedge^1[2].\]
Recall that a conformal structure is a nondegenerate section $\eta_{ab}$
of this bundle $\bigodot^2\!\Wedge^1[2]$
and hence we obtain a conformally invariant first-order linear differential
operator
\begin{equation}\label{tricky}\textstyle
TM\ni X^a\longmapsto{\mathcal{L}}_X\eta_{ab}\in\bigodot^2\!\Wedge^1[2].
\end{equation}
But recall that $\epsilon_{ab\dots d}$ is the canonical section of
$\Wedge^n[n]=\Wedge^0$ corresponding to the constant function~$1$. It follows
that ${\mathcal{L}}_X\epsilon_{ab\dots d}=0$ for any vector field $X^a$ and
therefore, from~(\ref{normalise}), that $\eta^{ab}{\mathcal{L}}_X\eta_{ab}=0$.
In addition, we may use the conformal metric $\eta_{ab}$ to lower indices, at
the expense of a conformal weight so that $TM=\Wedge^1[2]$. The Lie derivative
in (\ref{tricky}) thus yields a conformal invariant differential operator
\[\textstyle\Wedge^1[2]\longrightarrow\bigodot_\circ^2\!\Wedge^1[2],\]
where $\circ$ denotes the trace-free part (a manifestly conformally invariant
notion). If $\phi_{b\dots d}$ is an $(n-1)$-form then
\[\epsilon_{ab\dots d}\eta^{be}\cdots\eta^{dg}\phi_{e\dots g}\]
is a $1$-form of conformal weight $n-2(n-1)=2-n$. In other words, the bundles
$\Wedge^{n-1}$ and $\Wedge^1[2-n]$ are canonically isomorphic. Similarly, we
have $\Wedge^n=\Wedge^0[-n]$ and so the exterior derivative
${\rm d}\colon \Wedge^{n-1}\to\Wedge^n$ may be viewed as an invariant differential operator
\[\Wedge^1[2-n]\to\Wedge^0[-n].\]
Together with the exterior derivative ${\rm d}\colon \Wedge^1\to\Wedge^2$, we now have
three conformally invariant linear differential operators defined on variously
weighted $1$-forms, namely
\begin{equation}\label{all_three}\textstyle
\Wedge^1[2]\to\bigodot_\circ^2\!\Wedge^1[2],
\qquad\Wedge^1[2-n]\to\Wedge^0[-n],\qquad\Wedge^1\to\Wedge^2.\end{equation}
As in Section~\ref{LC}, we are now in a position to define a preferred
connection in the presence of a scale $\sigma\in\Gamma(\Wedge^0[1])$.
Specifically, we may drop all the weights in (\ref{all_three}) and combine
them to obtain
\[\textstyle\Wedge^1\xrightarrow{\,\nabla\,}\Wedge^1\otimes\Wedge^1
=\bigodot_\circ^2\!\Wedge^1\oplus\Wedge^0\oplus\Wedge^2.\]
The three bundles on the right constitute the orthogonal decomposition of
$\Wedge^1\otimes\Wedge^1$ into irreducibles with respect to the metric
$g_{ab}=\sigma^{-2}\eta_{ab}$. Of course, it is easily verified that we have
found the Levi-Civita connection for $g_{ab}$ but there are two advantages of
this particular route. One is that we are able to read off the change in the
Levi-Civita connection under conformal rescaling
$\sigma\mapsto\widehat\sigma=\Omega^{-1}\sigma$. The other is that this route
may be employed in other {\em parabolic geometries}. Regarding the
first advantage, another way of pinning down the connection defined by
(\ref{all_three}) and a scale $\sigma\in\Gamma\big({\Wedge}^0[1]\big)$, is to require
that $\nabla_a\sigma=0$. This is consistent with the connections induced on
the line bundles $\Wedge^0[w]$ in the sense that, for
$\phi\in\Gamma\big({\Wedge}^0[w]\big)$, we have $\nabla\phi=\sigma^w{\rm d}(\sigma^{-w}\phi)$.
In particular, for a new scale $\widehat\sigma=\Omega^{-1}\sigma$ and
$\phi\in\Gamma\big({\Wedge}^0[w]\big)$, we have
\[\widehat\nabla\phi=\widehat\sigma^w {\rm d}\big(\widehat\sigma^{-w}\phi\big)=
\Omega^{-w}\sigma^w{\rm d}\big(\Omega^w\sigma^{-w}\phi\big)
=\sigma^w{\rm d}\big(\sigma^{-w}\phi\big)+w\Omega^{-1}({\rm d}\Omega)\phi
=\nabla\phi+w\Upsilon\phi,\]
where $\Upsilon\equiv\Omega^{-1}{\rm d}\Omega$. It follows that, regarding the
change in the Levi-Civita connection on a~$1$-form~$\omega_b$, we have
\[\widehat\nabla_a\omega_b
=\nabla_a\omega_b-\Upsilon_a\omega_b-\Upsilon_b\omega_a
+\Upsilon^c\omega_cg_{ab}\]
since then, if $\omega_b$ has conformal weight~$w$, we deduce that
\[\widehat\nabla_a\omega_b
=\nabla_a\omega_b+(w-1)\Upsilon_a\omega_b-\Upsilon_b\omega_a
+\Upsilon^c\omega_cg_{ab},\]
which is exactly so that the three operators (\ref{all_three}), namely
\[\textstyle
\phi_b\mapsto\nabla_{(a}\phi_{b)}-\frac1n(\nabla^c\phi_c)g_{ab},\qquad
\phi_b\mapsto\nabla^b\phi_b,\qquad\phi_b\mapsto\nabla_{[a}\phi_{b]},\]
are invariantly defined.

Regarding the second advantage, it is convenient to adopt a universal notation
for the natural irreducible vector bundles available on a parabolic geometry,
as detailed in~\cite{BE,CS}. Let's see how this works in $5$-dimensional
conformal geometry. The general irreducible spin bundle is
\[\bthree{a}{b}{c},\qquad\mbox{where}\ a\in{\mathbb{R}}\
\mbox{and}\ b,c\in{\mathbb{Z}}_{\geq 0}\]
and for tensor bundles we restrict $c$ to be even. (Roughly speaking, this
indicates a (complex) irreducible representation of
${\mathbb{R}}_{>0}\times{\mathrm{Spin(5)}}$ with the real number $a$ over the
crossed node recording the action of~${\mathbb{R}}_{>0}$.) The de~Rham complex
is
\begin{equation}\label{deRham5dim}
\bthree{0}{0}{0}\to\bthree{-2}{1}{0}\to\bthree{-3}{0}{2}
\to\bthree{-4}{0}{2}\to\bthree{-5}{1}{0}\to\bthree{-5}{0}{0}\end{equation}
and the three conformally invariant first-order linear differential operators
on weighted $1$-forms are
\[\bthree{0}{1}{0}\to\bthree{-2}{2}{0},\qquad
\bthree{-5}{1}{0}\to\bthree{-5}{0}{0},\qquad
\bthree{-2}{1}{0}\to\bthree{-3}{0}{2}.\]
Notice that $\Wedge^0[w]=\bthree{w}{0}{0}$ and that the identification
$\Wedge^5=\Wedge^0[-5]$ is built into the notation. There are many other
advantages of this seemingly arcane notation: its utility in $5$-dimensional
spin geometry is the subject of the following section.

\section{Five-dimensional conformal spin geometry}\label{5d_conformal}
In this section we discuss five-dimensional conformal geometry by means of
spinors (that is to say four-spinors in this dimension (also known as
{\em twistors})). For simplicity, we shall suppose that the conformal metric
has split signature. There is some elementary linear algebra behind this
discussion as follows. Suppose that ${\mathbb{T}}$ is a four-dimensional real
vector space equipped with a nondegenerate skew form $\epsilon_{\alpha\beta}$.
Let us write $\epsilon^{\alpha\beta}$ for its inverse so that
$\epsilon_{\alpha\beta}\epsilon^{\alpha\gamma}=\delta_\alpha{}^\gamma$. The
vector space $\Wedge^2{\mathbb{T}}$ naturally splits:
\begin{equation}\label{symplectic_decomposition}
\Wedge^2{\mathbb{T}}
=\big\{P_{\alpha\beta}\,|\,\epsilon^{\alpha\beta}P_{\alpha\beta}=0\big\}
\oplus\{P_{\alpha\beta}=\lambda\epsilon_{\alpha\beta}\}
\equiv\Wedge_\perp^2{\mathbb{T}}\oplus{\mathbb{R}}.\end{equation}
Otherwise said, if we let ${\mathrm{Sp}}(4,{\mathbb{R}})$ denote the linear
automorphisms of ${\mathbb{T}}\cong{\mathbb{R}}^4$ preserving the symplectic
form $\epsilon_{\alpha\beta}$, then~(\ref{symplectic_decomposition}) is the
decomposition of $\Wedge^2{\mathbb{T}}$ into
${\mathrm{Sp}}(4,{\mathbb{R}})$-irreducibles. The $5$-dimensional vector space
$\Wedge_\perp^2{\mathbb{T}}$ is acquires a split signature metric
\[\|P_{\alpha\beta}\|^2\equiv
P_{\alpha\beta}P_{\gamma\eta}\epsilon^{\alpha\gamma}\epsilon^{\beta\eta}.\]
Otherwise said, we have constructed the isomorphism
${\mathrm{Sp}}(4,{\mathbb{R}})\cong{\mathrm{Spin}}(3,2)$. According to the
Pl\"ucker relations, the null vectors in $\Wedge_\perp^2{\mathbb{T}}$ are
the decomposable tensors.

With the conventions of~\cite{BE}, the bundle version of this discussion yields
the splitting
\[\Wedge^2\big(\bthree{0}{0}{1}\big)=\bthree{0}{1}{0}\oplus\bthree{1}{0}{0}
=\Wedge^1[2]\oplus\Wedge^0[1].\]
More precisely, we may view a split signature conformal spin structure on a
five-dimensional split signature conformal manifold as a rank $4$ bundle,
denoted by $\bthree{0}{0}{1}$ and equipped with
\begin{itemize}\itemsep=0pt
\item a nondegenerate section of $\Wedge^2\big(\bthree{0}{0}{1}\big)$ with values
in~$\Wedge^0[-1]$,
\item an identification
$\Wedge_\perp^2\big(\bthree{0}{0}{1}\big)=\Wedge^1[2]$.
\end{itemize}
A scale $\sigma\in\Gamma\big({\Wedge}^0[1]\big)$ induces a nondegenerate section of
$\Wedge^2\big(\bthree{0}{0}{1}\big)$ (and, according to the discussion in Section~\ref{conf}, a~compatible nondegenerate (split signature) symmetric form on
$\Wedge^1$).
The first summand in the splitting
\begin{equation}\label{adjoint_bundle}
\End\big(\bthree{0}{0}{1}\big)=\bthree{-1}{0}{2}\oplus\bthree{-1}{1}{0}
\oplus\bthree{0}{0}{0}\end{equation}
captures the trace-free endomorphisms of $\bthree{0}{0}{1}$ that preserve its
conformal skew form \big(and, simultaneously, the second summand in
\[\End\big({\Wedge}^1\big)=\End\big(\, \bthree{-2}{1}{0}\big)
=\bthree{-2}{2}{0}\oplus\bthree{-1}{0}{2}\oplus\bthree{0}{0}{0}\]
plays the same r\^ole with respect to the conformal metric on $\Wedge^1$\big).
{From} this viewpoint, a scale $\sigma\in\Gamma\big({\Wedge}^0[1]\big)$ gives rise to a
connection on the spin bundle $\bthree{0}{0}{1}$ as follows. We firstly insist
that this connection preserve its conformal skew form and also that the induced
connection on~$\Wedge^0[1]$ annihilate~$\sigma$. According to~(\ref{adjoint_bundle}), the freedom in choosing such a connection lies in
\begin{equation}\label{freedom}
\Wedge^1\otimes\bthree{-1}{0}{2}
=\bthree{-2}{1}{0}\otimes\bthree{-1}{0}{2}.\end{equation}
On the other hand, the induced operator $\nabla\colon \Wedge^1\to\Wedge^2$ differs
from the exterior derivative by a homomorphism $\Wedge^1\to\Wedge^2$ (it is the
{\em torsion} of the induced affine connection) and, according
to~(\ref{deRham5dim}), lies in
\begin{equation}\label{torsion}
\Hom\big({\Wedge}^1,\Wedge^2\big)=\bthree{0}{1}{0}\otimes\bthree{-3}{0}{2}.
\end{equation}
Comparing the bundles (\ref{freedom}) and (\ref{torsion}) and noting that they
are canonically isomorphic suggests that the homomorphism $\Wedge^1\to\Wedge^2$
may be precisely eliminated by the allowed freedom. It is easy to check that
this is, indeed, the case.

\section{Three-dimensional conformal spin geometry}\label{3d_conformal}
A split signature three-dimensional conformal spin structure may be viewed as a
rank two `spin bundle' $S$ equipped with an identification
$\bigodot^2\!S=\Wedge^1[2]$ (cf.~\cite{S}). With the Dynkin diagram notation
from~\cite{BE},
\[\textstyle S=\btwo{1}{0},\qquad \bigodot^2\!\btwo{1}{0}=\btwo{2}{0},\qquad
\Wedge^1=\btwo{2}{-2},\qquad\Wedge^0[1]=\btwo{0}{1},\]
and the de~Rham complex is
\[\Wedge^0=\btwo{0}{0}\to\btwo{2}{-2}\to\btwo{2}{-3}\to\btwo{0}{-3}
=\Wedge^3.\]
The conformal structure can now be characterised by decreeing that the simple
spinors in $\btwo{2}{0}$ are the null vectors in $\Wedge^1[2]=\btwo{2}{0}$.

The first summand in
\[\End\big(\btwo{1}{0}\big)=\btwo{2}{-1}\oplus\btwo{0}{0}\]
captures the trace-free endomorphisms of $\btwo{1}{0}$ and it follows that the
freedom in choosing a connection on this spin bundle annihilating a scale,
i.e., a~nowhere vanishing section of $\Wedge^2\big(\btwo{1}{0}\big)=\btwo{0}{1}$, lies
in
\[\Wedge^1\otimes\btwo{2}{-1}=\btwo{2}{-2}\otimes\btwo{2}{-1}.\]
On the other hand the torsion of the induced affine connection lies in
the canonically isomorphic bundle
\[\Hom\big({\Wedge}^1,\Wedge^2\big)=\btwo{2}{0}\otimes\btwo{2}{-3}.\]
It is easy to check that the freedom in choice of connection on $\btwo{1}{0}$
may be used exactly to eliminate this torsion. More precisely, we have the
following:
\begin{prop}\label{one} Given a scale, i.e., a nowhere vanishing
$\sigma\in\Gamma\big(\btwo{0}{1}\big)$, there is a unique connection on $\btwo{1}{0}$
so that
\begin{itemize}\itemsep=0pt
\item the induced connection on $\Wedge^2\big(\btwo{1}{0}\big)=\btwo{0}{1}$
annihilates~$\sigma$,
\item the torsion of the induced connection on $\Wedge^1=\btwo{2}{-2}$
vanishes.
\end{itemize}
\end{prop}
It is straightforward to figure out how this preferred connection changes under
change of scale. To do this, let us adapt the classical two-spinor notation
of~\cite{OT} to write
\[\nabla\colon \ \btwo{1}{0}\longrightarrow\btwo{2}{-2}\otimes\btwo{1}{0}
\quad\mbox{as} \ \phi_C\longmapsto\nabla_{AB}\phi_C.\]
\begin{prop}\label{two} Let us change scale $\sigma\in\Gamma\big({\Wedge}^0[1]\big)$ by
$\widehat\sigma=\Omega^{-1}\sigma$. Then, for $\phi_C\in\Gamma\big(\btwo{1}{0}\big)$,
\begin{equation}\label{change_of_spin_connection}
\widehat\nabla_{AB}\phi_C=\nabla_{AB}\phi_C+\Upsilon_{AB}\phi_C
-\Upsilon_{C(A}\phi_{B)},\qquad\mbox{where} \
\Upsilon_{AB}\equiv\Omega^{-1}\nabla_{AB}\Omega.\end{equation}
\end{prop}
\begin{proof} Given $\nabla_{AB}$ we use (\ref{change_of_spin_connection}) to
define $\widehat\nabla_{AB}$ and then verify that it has the characterising
properties from Proposition~\ref{one} for the scale~$\widehat\sigma$. Firstly,
if $\sigma_{CD}\in\Gamma\big({\Wedge}^2\big(\btwo{1}{0}\big)\big)$, then
\begin{gather}\label{weight1}
\widehat\nabla_{AB}\sigma_{CD}=\nabla_{AB}\sigma_{CD}
+2\Upsilon_{AB}\sigma_{CD}-\Upsilon_{C(A}\sigma_{B)D}
+\Upsilon_{D(A}\sigma_{B)C}=\nabla_{AB}\sigma_{CD}
+\Upsilon_{AB}\sigma_{CD}.\!\!\!\!\!\end{gather}
Therefore $\widehat\nabla_{AB}\widehat\sigma_{CD}
=\nabla_{AB}\big(\Omega^{-1}\sigma_{CD}\big) +\Upsilon_{AB}\big(\Omega^{-1}\sigma_{CD}\big)
=\Omega^{-1}\nabla_{AB}\sigma_{CD}$. Thus, if
$\nabla_{AB}\sigma_{CD}=0$, then $\widehat\nabla_{AB}\widehat\sigma_{CD}=0$,
which accounts for the first condition in Proposition~\ref{one}. Now equation~(\ref{weight1}) may be abbreviated as
$\widehat\nabla_{AB}\sigma=\nabla_{AB}\sigma+\Upsilon_{AB}\sigma$ for
$\sigma\in\Gamma\big(\btwo{0}{1}\big)$ and it follows that
\[\widehat\nabla_{AB}\rho=\nabla_{AB}\rho+w\Upsilon_{AB}\rho,\qquad
\mbox{for}\ \rho\in\Gamma\big(\btwo{0}{w}\big).\]
Combining this with (\ref{change_of_spin_connection}), it follows
that $\widehat\nabla_{AB}\phi_C=\nabla_{AB}\phi_C-\Upsilon_{C(A}\phi_{B)}$,
for $\phi_C\in\Gamma\big(\btwo{1}{-1}\,\big)$ and hence that
\[\widehat\nabla_{AB}\omega_{CD}=\nabla_{AB}\omega_{CD}
-\Upsilon_{C(A}\omega_{B)D}-\Upsilon_{D(A}\omega_{B)C},\]
for $\omega_{CD}\in\Gamma\big(\btwo{2}{-2}\,\big)=\Gamma\big({\Wedge}^1\big)$. It follows that
\[\widehat\nabla_{AB}\omega_{CD}-\widehat\nabla_{CD}\omega_{AB}
=\nabla_{AB}\omega_{CD}-\nabla_{CD}\omega_{AB},\]
which accounts for the second condition in Proposition~\ref{one}.
\end{proof}

\section{The Rumin complex}\label{rumin}
Contact geometry is the geometry of a \textit{maximally non-integrable}
corank one subbundle $H\subset TM$, where $M$ is of dimension $2n+1$. Maximal
non-integrability is to say that locally $H$ is given as the kernel of a
\textit{contact form} $\alpha$ such that $\alpha \wedge ({\rm d}\alpha)^n$ is
non-vanishing.

Alternatively, writing $\Wedge_H^1$ for the bundle of one
forms $\Wedge^1$ restricted naturally to $H$, and $L$ for the annihilator of
$H$, we have a short exact sequence
\begin{equation*}
0\to L\to\Wedge^1\to\Wedge_H^1\to 0
\end{equation*}
and, therefore, short exact sequences
\begin{equation}\label{SES2}
0\to \Wedge_H^{k-1}\otimes L\to\Wedge^k\to\Wedge_H^k\to 0\end{equation}
for $k=1,\ldots,2n$.
If we now consider the exterior derivative
\[\begin{array}{ccc}\Wedge_H^1&&\Wedge_H^2\\ \uparrow&&\uparrow\\
\Wedge^1&\xrightarrow{\ {\rm d}\ }\hspace{-14pt}
&\Wedge^2\\ \uparrow&&\uparrow\\
L&&\Wedge_H^1\otimes L,
\end{array}\]
then, by the Leibniz rule, the composition
$L\to\Wedge^1\xrightarrow{\,{\rm d}\,}\Wedge^2\to\Wedge_H^2$ is actually a vector
bundle homomorphism
known as the {\em Levi form}. The maximal non-integrability condition is
equivalent to the Levi form being injective with image consisting of
nondegenerate forms.

If one writes out the de Rham sequence along with the short exact
sequences (\ref{SES2}), one can obtain by diagram chasing, the Rumin
complex~\cite{R}. The Rumin complex is a replacement for the de Rham complex
on any contact manifold in that it computes the de Rham cohomology, but it is
in some sense more efficient in that derivatives are only taken in contact
directions. We are concerned with the case $\dim M = 5$ in which case the
diagram to chase is
\[\begin{array}{@{}ccccccccccc}
&&\Wedge_H^1&&\Wedge_H^2&&\Wedge_H^3&&\Wedge_H^4&&\\
&&\uparrow&&\uparrow&&\uparrow&&\uparrow&&\\
\Wedge^0 & \hspace{3mm}\xrightarrow{\ {\rm d}\ }
\hspace{3mm} & \Wedge^1&\hspace{3mm}\xrightarrow{\ {\rm d}\ }&\Wedge^2
&\xrightarrow{\ {\rm d}\ }&\Wedge^3
&\xrightarrow{\ {\rm d}\ }&\Wedge^4
& \xrightarrow{\ {\rm d}\ }& \hspace{3mm}\Wedge^5\\\
&&\uparrow&&\uparrow&&\uparrow&&\uparrow&&\\
&&L&&\Wedge_H^1\otimes L && \Wedge_H^2\otimes L && \Wedge_H^3\otimes L,&&
\end{array}\]
and, writing $\Wedge_H^2=\Wedge_{H\perp}^2\oplus L$, where $\Wedge_{H\perp}^2$
comprises $2$-forms on $H$ that are trace-free with respect to the Levi form,
one obtains the invariantly defined complex
\[\begin{array}{@{}ccccccccccc}
\Wedge^0 & \hspace{2mm} \xrightarrow{\ {\rm d}_\perp \ }
& \Wedge^1_H&\xrightarrow{\ {\rm d}_\perp \ }
&\Wedge^2_{H\perp}&\xrightarrow{\ {\rm d}_\perp^{(2)} \ }
&\Wedge_{H\perp}^2 \otimes L&\xrightarrow{\ {\rm d}_\perp \ }
&\Wedge^3_H \otimes L & \xrightarrow{\ {\rm d}_\perp \ }
& \hspace{2mm} \Wedge^5.
\end{array}\]
A difference to the de Rham complex is that one obtains
${\rm d}_\perp^{(2)}\colon \Wedge_{H\perp}^2 \to \Wedge_{H\perp}^2 \otimes L$,
which is a~second-order differential operator.

\section[Spinors in G\_2 contact geometry]{Spinors in $\boldsymbol{G_2}$ contact geometry}\label{G2}
A $G_2$ contact geometry is an additional structure on the contact distribution
of a five-dimensional contact manifold. As observed in the previous section, a~contact geometry is naturally equipped with its Levi form $L\to\Wedge_H^2$ and
the contact distribution $H$ thereby inherits a nondegenerate skew form
defined up to scale. This is just what is needed to talk about {\em Legendrean
varieties}~\cite{LM} in the projective bundle ${\mathbb{P}}(H)\to M$. A~{\em $G_2$ contact} structure on $M$ is a field of Legendrean twisted cubics
in ${\mathbb{P}}(H)$. Precisely, this means that, for all $m\in M$, there is a
twisted cubic $C_m\subset{\mathbb{P}}(H_m)$, varying smoothly with $m\in M$,
such that the $2$-planes in $H_m$ covering the tangent lines to the cubic are
null for the Levi form. Equivalently, such a $G_2$ contact structure may be
viewed as a~rank two `spin bundle' S equipped with a `Levi-compatible'
identification $\bigodot^3\!S=\Wedge_H^1[2]$, where $\Wedge^5=\Wedge^0[-3]$.
Levi-compatibility means that the Levi form
\[\textstyle
L[4]\to\Wedge_H^2[4]=\Wedge^2\big({\bigodot}^3\!S\big)
=\big({\bigodot}^4\!S\otimes\Wedge^2S\big)\oplus\big({\Wedge}^2S\big)^3\]
has its range in the second summand and thus provides an identification
$L[4]=\big({\Wedge}^2S\big)^3$. In these circumstances, notice that
\[\textstyle\Wedge_H^4[8]=\Wedge^4\big({\bigodot}^3\!S\big)=\big({\wedge}^2S\big)^6\quad
\Rightarrow\quad
\Wedge^0[9]=\Wedge^5[12]=\Wedge_H^4[8]\otimes L[4]=
\big({\Wedge}^2S\big)^9\]
and, therefore, we find canonical identifications $\Wedge^2S=\Wedge^0[1]$ and
$L=\Wedge^0[-1]$. In any case, the $G_2$ contact structure can now be
characterised by decreeing that the simple spinors in $H=\bigodot^3\!S[-1]$
constitute the cone over the twisted cubic (and there is a clear analogy with
conformal spin geometry in dimension three). More detail can be found in
\cite[Section~5]{EN2} and the flat model is presented in~\cite[Section~4]{EN1}. We should
also point out that the geometry of the rank four bundle~$H$ follows that of
Bryant's `$H_3$-structures' on the tangent bundle in four dimensions~\cite{B}.

The reason for the name `$G_2$ contact structure' is that this geometric data
defines a parabolic geometry of type $(G_2,P)$ where~$G_2$ is the
simply-connected exceptional Lie group of split type~$G_2$ and~$P$ is a~particular parabolic subgroup such that~$G_2/P$ is a contact manifold: see, for
example,~\cite[Section~4.2.8]{CS} (and, in particular, it is explained in~\cite{EN1}
that this particular realisation of the Lie algebra of $G_2$ goes back to
Engel~\cite{E}). With the Dynkin diagram notation from~\cite{BE}, this
motivates our writing $\bigodot^k\!S[w]=\gtwo{k}{w}$ (just another way of
organising the irreducible representations of
${\mathrm{GL}}_+(2,{\mathbb{R}})$) so that
\begin{gather*} \textstyle S=\gtwo{1}{0},\qquad \bigodot^3\!\gtwo{1}{0}=\gtwo{3}{0},\qquad
\Wedge_H^1=\gtwo{3}{-2},\qquad H=\gtwo{3}{-1},\qquad\Wedge^0[1]=\gtwo{0}{1},
\end{gather*}
and the Rumin complex is
\[\gtwo{0}{0}\xrightarrow{\,{\rm d}_\perp\,}\gtwo{3}{-2}
\xrightarrow{\,{\rm d}_\perp\,}\gtwo{4}{-3}
\xrightarrow{\,{\rm d}_\perp^{(2)}\,}\gtwo{4}{-4}
\xrightarrow{\,{\rm d}_\perp\,}\gtwo{3}{-4}\to
\gtwo{0}{-3}\]
(consistent with the basic {\em BGG complex} from~\cite{BE}). Notice that
\[\Wedge_H^2=\Wedge^2\big(\gtwo{3}{-2}\,\big)
=\gtwo{4}{-3}\oplus\gtwo{0}{-1}\]
so that the Levi form $L=\gtwo{0}{-1}\hookrightarrow\Wedge^2\big({\Wedge}_H^1\big)$ is
built into the notation.

Regarding calculus on a contact manifold, it is natural to consider
{\em partial connections}, rather than connections, on vector bundles in which
directional derivatives are defined, in the first instance, only in the contact
directions. More precisely, a {\em partial connection} on a vector bundle~$E$
is a~linear differential operator
\[\nabla_H\colon \ E\to\Wedge_H^1\otimes E\]
satisfying a {\em partial Leibniz rule} $\nabla_H(fs)=f\nabla_Hs+{\rm d}_\perp
f\otimes s$. (In fact, a partial connection can be uniquely promoted
\cite[Proposition~3.5]{EG} to a full connection but we shall not need this
trick.) In analogy with three-dimensional spin geometry, we may construct a~preferred partial connection on~$\gtwo{1}{0}$ in the presence of a~`scale'
$\sigma\in\Gamma\big(\gtwo{0}{1}\big)$.

The construction of this preferred partial connection follows the same route
save for a minor yet crucial distinction. For any contact manifold of dimension
$\geq 5$, a partial connection on $\Wedge_H^1$ gives rise to a linear
differential operator $\nabla_\perp\colon \Wedge_H^1\to\Wedge_{H\perp}^2$ defined as
the composition
\[\Wedge_H^1\xrightarrow{\,\nabla_H\,}
\Wedge_H^1\otimes\Wedge_H^1\xrightarrow{\,\wedge\,}
\Wedge_H^2\to\Wedge_{H\perp}^2\]
with the same symbol as the invariantly defined Rumin operator. It follows
that the difference
\[\nabla_\perp-{\rm d}_\perp\colon \ \Wedge_H^1\to\Wedge_{H\perp}^2\]
is actually a homomorphism of vector bundles. By definition, this is the
{\em partial torsion} of a~partial connection
$\nabla_H\colon \Wedge_H^1\to\Wedge_H^1\otimes\Wedge_H^1$.

In the case of a $G_2$ contact structure, bearing in mind that
$\Wedge^2\big(\gtwo{1}{0}\big)=\gtwo{0}{1}$, a partial connection on the spin bundle
$\gtwo{1}{0}$ induces partial connections on all spin bundles $\gtwo{k}{w}$
and, in particular, on $\Wedge_H^1=\gtwo{3}{-2}$. Thus, we may ask about
its partial torsion, which lies in
\begin{equation}\label{partial_torsion_decomposition}
\Hom\big({\Wedge}_H^1,\Wedge_{H\perp}^2\big)=\Hom\big(\gtwo{3}{-2},\gtwo{4}{-3}\,\big)
=\gtwo{7}{-4}\oplus\gtwo{5}{-3}\oplus\gtwo{3}{-2}\oplus\gtwo{1}{-1}.
\end{equation}
This decomposition is crucial in characterising preferred spin connections as
follows.
\begin{prop}\label{three}
Given a scale, i.e., a nowhere vanishing $\sigma\in\Gamma\big(\gtwo{0}{1}\big)$, there
is a unique partial connection on $\gtwo{1}{0}$ so that
\begin{itemize}\itemsep=0pt
\item the induced partial connection on $\Wedge^2\big(\gtwo{1}{0}\big)=\gtwo{0}{1}$
annihilates~$\sigma$,
\item the partial torsion of the induced partial connection on $\Wedge_H^1=\gtwo{3}{-2}$ lies in $\gtwo{7}{-4}$.
\end{itemize}
\end{prop}
\begin{proof}
The first summand in
\[\End\big(\gtwo{1}{0}\big)=\gtwo{2}{-1}\oplus\gtwo{0}{0}\]
captures the trace-free endomorphisms of $\gtwo{1}{0}$ and it follows that the
freedom in choosing a~partial connection on this spin bundle annihilating
$\sigma$ lies in
\begin{equation}\label{partial_freedom}
\Wedge_H^1\otimes\gtwo{2}{-1}=\gtwo{3}{-2}\otimes\gtwo{2}{-1}
=\gtwo{5}{-3}\oplus\gtwo{3}{-2}\oplus\gtwo{1}{-1}.\end{equation}
Comparison with (\ref{partial_torsion_decomposition}) certainly suggests that
all but the piece in $\gtwo{7}{-4}$ can be uniquely eliminated. We may verify
this using spinors. With the familiar conventions of~\cite{OT}, let us write
$\epsilon_{AB}\in\Gamma\big({\Wedge}^2\big(\gtwo{1}{0}\big)\big)$ rather than
$\sigma\in\Gamma\big(\gtwo{0}{1}\big)$. Then,
choosing a partial connection
\[\phi_D\xrightarrow{\,\nabla_H\,}\nabla_{ABC}\phi_D\qquad\mbox{on} \quad
\gtwo{1}{0},\]
the general partial connection on $\gtwo{1}{0}$ with
$\nabla_{ABC}\epsilon_{DE}=0$ has the form
\[\nabla_{ABC}\phi_D+\Gamma_{ABCD}{}^E\phi_E,\]
where $\Gamma_{ABCDE}=\Gamma_{(ABC)(DE)}$ (i.e., lying in
$\gtwo{3}{-2}\otimes\gtwo{2}{-1}$, as in~(\ref{partial_freedom})). By the
Leibniz rule, the induced
operator $\gtwo{3}{-2}\to\gtwo{4}{-3}$ is
\begin{equation}\label{induced_operator}
\omega_{DEF}\longmapsto\nabla_{(AB}{}^E\omega_{CD)E}
+2\Gamma_{(AB}{}^E{}_C{}^G\omega_{D)EG}
-\Gamma_{E(AB}{}^E{}^G\omega_{CD)G}.\end{equation}
Therefore, according to the decomposition~(\ref{partial_freedom}), we should
now write
\begin{equation}\label{split}\Gamma_{ABC}{}^{DE}=\lambda_{ABC}{}^{DE}
+\mu_{(AB}{}^{(D}\delta_{C)}{}^{E)}
+\nu_{(A}\delta_B{}^D\delta_{C)}{}^E,\end{equation}
where $\lambda_{ABCDE}$ and $\mu_{ABC}$ are symmetric spinors and compute
\[2\Gamma_{(AB}{}^E{}_C{}^G\omega_{D)EG}
-\Gamma_{E(AB}{}^E{}^G\omega_{CD)G}\]
for each term on the right hand side of~(\ref{split}). Clearly, this entails
computing
\begin{equation}\label{entrails}
\Gamma_{(AB}{}^{(E}{}_{C)}{}^{G)}\qquad\mbox{and}\qquad
\Gamma_{EAB}{}^{EG}.\end{equation}

Firstly, if $\Gamma_{ABC}{}^{DE}=\lambda_{ABC}{}^{DE}$, where
$\lambda_{ABCDE}=\lambda_{(ABCDE)}$, then the second term in~(\ref{entrails})
vanishes so for (\ref{induced_operator}) we end up with
\[\omega_{DEF}\longmapsto\nabla_{(AB}{}^E\omega_{CD)E}
+2\lambda_{(ABC}{}^{EG}\omega_{D)EG},\]
which is perfect for eliminating the $\gtwo{5}{-3}$-component of partial
torsion.

Secondly, if $\Gamma_{ABC}{}^{DE}=\mu_{(AB}{}^{(D}\delta_{C)}{}^{E)}$, where
$\mu_{ABC}=\mu_{(ABC)}$, then straightforward spinor computations show that
\[\textstyle\Gamma_{(AB}{}^{(E}{}_{C)}{}^{G)}
=\frac16\mu_{(AB}{}^{(E}\delta_{C)}{}^{G)}\qquad\mbox{and}\qquad
\Gamma_{EAB}{}^{EG}=\frac56\mu_{AB}{}^G\]
so for (\ref{induced_operator}) we end up with
\begin{align*} \omega_{DEF}& \mapsto
\nabla_{(AB}{}^E\omega_{CD)E}
+\tfrac13\mu_{(AB}{}^E\omega_{CD)E}
-\tfrac56\mu_{(AB}{}^G\omega_{CD)G}\\
& =\nabla_{(AB}{}^E\omega_{CD)E}
-\tfrac12\mu_{(AB}{}^E\omega_{CD)E},\end{align*}
which is perfect for eliminating the $\gtwo{3}{-2}$-component of partial
torsion.

Thirdly, if $\Gamma_{ABC}{}^{DE}=\nu_{(A}\delta_B{}^D\delta_{C)}{}^E$, then
straightforward spinor calculations yield
\[\textstyle\Gamma_{(AB}{}^{(E}{}_{C)}{}^{G)}
=-\frac13\nu_{(A}\delta_B{}^E\delta_{C)}{}^G\qquad\mbox{and}\qquad
\Gamma_{EAB}{}^{EG}=\frac43\nu_{(A}\delta_{B)}{}^G\]
so for~(\ref{induced_operator}) we end up with
\[\textstyle\omega_{DEF}\mapsto\nabla_{(AB}{}^E\omega_{CD)E}
-\frac23\nu_{(A}\omega_{BCD)}
-\frac43\nu_{(A}\omega_{BCD)}
=\nabla_{(AB}{}^E\omega_{CD)E}
-2\nu_{(A}\omega_{BCD)},\]
which is perfect for eliminating the $\gtwo{1}{-1}$-component of partial
torsion.
\end{proof}

Several remarks are in order. Firstly, the preferred connection of
Proposition~\ref{three} is constructed by eliminating all but the
$\gtwo{7}{-4}$-component of the partial torsion of the induced partial
connection on $\Wedge_H^1$, decomposed according to
(\ref{partial_torsion_decomposition}). In fact, it is clear from the proof that
the component lying in $\gtwo{7}{-4}$ is the same for any choice of partial
connection on $\gtwo{1}{0}$ and is, therefore, an invariant of the structure.
It is called the {\em torsion} of our $G_2$ contact structure. In the general
theory of parabolic geometry~\cite{CS}, this is the only component of
{\em harmonic curvature} and is therefore the only obstruction to
{\em local flatness}, i.e., to being locally isomorphic to the flat
model~$G_2/P$. Secondly, we should point out that the spinor identities
established by direct calculation in our proof can be avoided by judicious use
of Lie algebra cohomology (as in done in~\cite{CS}). Thirdly, we note that a
{\em scale}, a nowhere vanishing section~$\sigma$ of~$\gtwo{0}{1}$, has a
nice geometric interpretation. Since $\gtwo{0}{-1}=L\hookrightarrow\Wedge^1$
is the bundle of contact forms, we can interpret~$\sigma^{-1}$ as a~choice of
contact form. In other words, the preferred partial connection on the spin
bundle $S=\gtwo{1}{0}$ is obtained in the presence of a contact form.

The transformation law for preferred partial connections in $G_2$ contact
geometry is obtained by analogy with Proposition~\ref{two}. Its proof will
therefore be omitted.
\begin{prop} Let us change scale $\sigma\in\Gamma\big({\Wedge}^0[1]\big)$ by
$\widehat\sigma=\Omega^{-1}\sigma$. Then, for $\phi_D\in\Gamma\big(\gtwo{1}{0}\big)$,
\begin{equation}\label{change_of_G2_spin_connection}
\widehat\nabla_{ABC}\phi_D=\nabla_{ABC}\phi_D+\Upsilon_{ABC}\phi_D
-\Upsilon_{D(AB}\phi_{C)},\qquad\mbox{where} \quad
\Upsilon_{ABC}\equiv\Omega^{-1}\nabla_{ABC}\Omega.\end{equation}
\end{prop}
As an immediate consequence of (\ref{change_of_G2_spin_connection}), if
$\phi_D\in\Gamma\big(\gtwo{1}{0}\big)$, then
$\widehat\nabla_{(ABC}\phi_{D)}=\nabla_{(ABC}\phi_{D)}$. Furthermore, if
$\phi_D\in\Gamma\big(\gtwo{1}{-4/3}\ \big)$, then
\[\textstyle\widehat\nabla_{ABC}\phi_D
=\nabla_{ABC}\phi_D-\frac13\Upsilon_{ABC}\phi_D-\Upsilon_{D(AB}\phi_{C)},\]
whence $\widehat\nabla_{AB}{}^C\phi_C=\nabla_{AB}{}^C\phi_C$. We have found
two {\em invariant operators}
\begin{equation}\label{basic_invariant_operators}
\gtwo{1}{0}\longrightarrow\gtwo{4}{-3}\qquad\mbox{and}\qquad
\gtwo{1}{-4/3}\;\longrightarrow\gtwo{2}{-7/3}\ ,\end{equation}
defined only in terms of the $G_2$ contact structure itself. Conversely, it is
easy to see that the existence of these two operators is sufficient to define
the preferred partial connection on $\gtwo{1}{0}$ associated with a scale and
to capture the transformation law~(\ref{change_of_G2_spin_connection}).
Sadly, we have not been able to manufacture either of the invariant operators
(\ref{basic_invariant_operators}) directly.

In three-dimensional conformal spin geometry, the transformation law
(\ref{change_of_spin_connection}) leads to a pair of basic first-order
invariant differential operators
\[\btwo{1}{0}\longrightarrow\btwo{3}{-2}\qquad\mbox{and}\qquad
\btwo{1}{-3/2}\;\longrightarrow\btwo{1}{-5/2}\ ,\]
given by $\phi_C\mapsto\nabla_{(AB}\phi_{C)}$ and
$\phi_C\mapsto\nabla_{AB}\phi^B$. This suggests that we should refer to the
opera\-tors~(\ref{basic_invariant_operators}) as the `twistor' or `tractor'
operator and `Dirac' operator, respectively, on a $G_2$ contact manifold. Sure
enough, this tractor operator is overdetermined and, in the flat case, has a
$7$-dimensional kernel corresponding to the embedding
$G_2\hookrightarrow{\mathrm{SO}}^\uparrow(4,3)$. The prolongation of this
tractor operator (leading to the {\em standard tractor} bundle) is detailed
in~\cite{Moy}.

\section{Legendrean contact geometry in five dimensions}\label{CL}
A Legendrean contact geometry in five dimensions is a $5$-dimensional contact
manifold equipped with a splitting of the contact distribution as two
rank~$2$ subbundles
\[H=E\oplus F\]
each of which is null with respect to the Levi form
(so that the Levi form reduces to a perfect pairing $E\otimes F\to L^*$). The
flat model naturally arises in \cite[Section~2]{EN1} as the moduli space of flying
saucers in `attacking mode' and \cite[Proposition ~2.5]{EN1} gives
the $15$-dimensional symmetry algebra.

In the spirit of previous sections our aim will be, in the presence of a~scale (equivalently, a~choice of contact form), to construct preferred partial
connections on all the natural irreducible bundles on such a geometry. As
before, these connections can be obtained by means of the canonical
differential operators present on this type of geometry. In fact, we shall
only need to examine the Rumin complex to find sufficiently many canonical
differential operators for these purposes.

A Legendrean contact geometry is a type of parabolic
geometry~\cite[Section~4.2.3]{CS}, specifically \raisebox{1pt}{$\xox{}{}{}$} in the
notation of~\cite{BE}. Then
\[H=\begin{array}{c}E\\ \oplus\\ F\end{array}=
\begin{array}{c}
\xox{1}{1}{-1}\\
\oplus\\
\xox{-1}{1}{1}.
\end{array}\]
The general irreducible bundle has the form $\xox{u}{k}{v}$ for
$k\in{\mathbb{Z}}_{\geq0}$ and $u,v,\in{\mathbb{R}}$. It is convenient to take
$S\equiv\xox{-1}{1}{-1}$ as our basic `spin' bundle. If we also let
$\Wedge^0[u,v]\equiv\xox{u}{0}{v}$, then $L=\Wedge^0[-1,-1]$ and the general
irreducible bundle is $\xox{u}{k}{v}=\bigodot^k\!S[u+k,v+k]$. Notice that
\[E\otimes F
=\xox{1}{1}{-1}\;\otimes\;\xox{-1}{1}{1}=\xox{0}{2}{0}\oplus\xox{1}{0}{1}\]
and the Levi form is built into the notation as projection onto the second
summand. \big(Without the Dynkin diagram notation, it follows from the perfect
pairing $E\otimes F\to L^*$ and the canonical identification $E^*=E\otimes
\det E^*$ that
\[E\otimes(\det E^*)^{1/2}=\big(E\otimes(\det E^*)^{1/2}\big)^*
=F\otimes(\det F^*)^{1/2}\]
and we may take $S\equiv E\otimes(\det E^*)^{1/2}\otimes L^{1/2}
\equiv F\otimes(\det F^*)^{1/2}\otimes L^{1/2}$.\big)

In order to avoid confusion, by default we shall write a section of
$\xox{u}{k}{v}$ with {\em lower} spinor indices
\[\phi_{\mbox{\scriptsize$\underbrace{AB\dots C}_k$}}=\phi_{(AB\dots C)}
\in\Gamma\big(\xox{u}{k}{v}\big)\]
with no special terminology to record the bundle of which it is a section (in
other words, we shall forgo any systematic notion of `weight'). Of course, we
may use the tautological identification $S=S^*\otimes\det S=S^*\otimes L$ to
replace lower spinor indices by upper spinor indices (with an appropriate
change in `weight' if we were to assign one) so there is no loss in using lower
indices by default. As an example of these conventions in action, we may write
the first operator ${\rm d}_\perp\colon \Wedge^0\to\Wedge_H^1$ in the Rumin complex as
\begin{equation}\label{first_rumin}
f\longmapsto {\rm d}_\perp f\equiv\left[\!\begin{array}{c}\nabla_Af\\
\bar\nabla_Af\end{array}\!\right]
\in\begin{array}{c}\xox{-2}{1}{0}\\[-3pt] \oplus\\ \xox{0}{1}{-2}\end{array}
=\begin{array}{c}E^*\\[-3pt] \oplus\\ F^*\end{array}=\Wedge^1_H,
\end{equation}
where $\nabla_A$, respectively~$\bar\nabla_A$, is the directional derivative in
the~$E$, respectively~$F$, direction, both of which are manifestly invariantly
defined. (Although the notion of Legendrean contact geometry pertains in any
odd dimension, it is only in five dimensions that one has the convenience of
spinors and, in particular, that the bundles $E$ and $F$ agree save for a line
bundle factor.) Of course, the directional derivatives $\nabla_Af$ and
$\bar\nabla_Af$ end up as sections of different bundles even though each of
them has a single spinor index.

Soon (as with all parabolic geometries~\cite{CSweyl}), we shall find it
convenient to work in a particular {\em scale}, i.e., with a nowhere vanishing
section $\sigma\in\Gamma\big(\xox{1}{0}{1}\big)$. As with all parabolic contact
structures~\cite[Section~4.2]{CS}, we may interpret the section $\sigma^{-1}$ of
$\,\xox{-1}{0}{-1}\,=L\hookrightarrow\Wedge^1$ as a choice of contact form.
For $5$-dimensional Legendrean contact geometry, however, we may also interpret
a scale as a choice of skew spinor $\epsilon_{AB}$ by dint of
\begin{equation}\label{epsilon}\sigma^{-1}\in\Gamma\big(\,\xox{-1}{0}{-1}\,\big)
=\Gamma\big(\Wedge^2\big(\,\xox{-1}{1}{-1}\,\big)\big)=\Gamma\big(\Wedge^2S\big)
\ni\epsilon_{AB},\end{equation} which we may use to raise and lower spinor
indices with the familiar conventions of~\cite{OT}. In particular, for
$\phi_A\in\Gamma\big(\,\xox{-1}{1}{-1}\,\big)$ and $\psi_A\in\Gamma\big(\xox{0}{1}{0}\big)$, we
may write
\[\textstyle\sigma^{-1}\phi_{[A}\psi_{B]}=\frac12\epsilon_{AB}\psi^C\phi_C
\qquad\mbox{to define }\psi^C\phi_C\in\Gamma\big(\Wedge^0\big),\]
independent of choice of $\sigma$ and identifying
$\big(\,\xox{-1}{1}{-1}\,\big)^*=\xox{0}{1}{0}$, as expected.

Continuing from (\ref{first_rumin}), we may decompose the entire Rumin complex
into its constituent parts via the usual spin-bundle decompositions to obtain
an array of invariantly defined linear differential operators
\begin{gather}\label{array_of_operators}
\raisebox{-40pt}{\setlength{\unitlength}{.95pt}
\begin{picture}(410,120)(0,-25)
\put(0,50){\makebox(0,0){\xox{0}{0}{0}}}
\put(80,70){\makebox(0,0){\xox{-2}{1}{0}}}
\put(80,30){\makebox(0,0){\xox{0}{1}{-2}}}
\put(160,90){\makebox(0,0){\xox{-3}{0}{1}}}
\put(160,50){\makebox(0,0){\xox{-2}{2}{-2}}}
\put(160,10){\makebox(0,0){\xox{1}{0}{-3}}}
\put(240,90){\makebox(0,0){\xox{-4}{0}{0}}}
\put(240,50){\makebox(0,0){\xox{-3}{2}{-3}}}
\put(240,10){\makebox(0,0){\xox{0}{0}{-4}}}
\put(320,70){\makebox(0,0){\xox{-4}{1}{-2}}}
\put(320,30){\makebox(0,0){\xox{-2}{1}{-4}}}
\put(400,50){\makebox(0,0){\xox{-3}{0}{-3}}}
\put(30,49){\vector(2,1){25}}
\put(30,43){\vector(2,-1){25}}
\put(110,69){\vector(2,1){25}}
\put(110,63){\vector(2,-1){25}}
\put(110,29){\vector(2,1){25}}
\put(110,23){\vector(2,-1){25}}
\put(190,86){\vector(1,0){25}}
\put(190,46){\vector(1,0){25}}
\put(190,6){\vector(1,0){25}}
\put(190,83){\vector(3,-4){25}}
\put(190,83){\vector(3,-4){25}}
\put(190,49){\vector(3,4){25}}
\put(190,43){\vector(3,-4){25}}
\put(190,9){\vector(3,4){25}}
\put(270,83){\vector(2,-1){25}}
\put(270,49){\vector(2,1){25}}
\put(270,43){\vector(2,-1){25}}
\put(270,9){\vector(2,1){25}}
\put(350,63){\vector(2,-1){25}}
\put(350,29){\vector(2,1){25}}
\put(202.5,-13){\makebox(0,0){\framebox{second-order operators}}}
\put(202.5,-4){\vector(0,1){7}}
\put(46,-13){\makebox(0,0){\framebox{first-order operators}}}
\put(42.5,-4){\vector(0,1){32}}
\put(100,-13){\line(1,0){20}}
\put(120,-13){\vector(0,1){20}}
\put(359,-13){\makebox(0,0){\framebox{first-order operators}}}
\put(362.5,-4){\vector(0,1){32}}
\put(305,-13){\line(-1,0){20}}
\put(285,-13){\vector(0,1){20}}
\end{picture}}\!\!\!\!\!
\end{gather}
In this diagram we have omitted arrows that correspond to homomorphisms. For
example one can check that the part of the Rumin complex
$\xox{-2}{1}{0}\to\xox{1}{0}{-3}$ is actually a homomorphism, and its vanishing
is equivalent to the integrability of~$F$. As already noted, the operator
\[\nabla_A\colon \ \xox{0}{0}{0}\to\xox{-2}{1}{0}\]
is just the directional derivative on functions in the $E$ direction:
$f \mapsto {\rm d}f|_E$.
More generally, the partial connections we aim to construct naturally split
into a part that differentiates along~$E$ (which we shall denote by~$\nabla_A$)
and a part that differentiates along~$F$ (denoted by~$\bar\nabla_A$). In
particular, amongst the invariant operators in (\ref{array_of_operators}) we
find $\nabla_A\colon \xox{0}{0}{-4}\to\xox{-2}{1}{-4}$ and, therefore, by insisting
on the Leibniz rule, invariantly defined derivatives in the $E$ direction
\[\nabla_A\colon \ \xox{0}{0}{v}\to\xox{-2}{1}{0}\otimes\xox{0}{0}{v},\qquad
\mbox{for all}\ v\in{\mathbb{R}}.\]
Now suppose we are given a nowhere vanishing {\em scale}
$\sigma\in\Gamma\big(\xox{1}{0}{1}\big)$. Then we can define
\begin{align}\label{densities}
\nabla_A\colon \ \xox{u}{0}{v} \to \xox{-2}{1}{0} \otimes \xox{u}{0}{v},\qquad
\mbox{for all}\ u,v\in{\mathbb{R}}
\end{align}
by $\nabla_A(f\sigma^u)\equiv(\nabla_A f)\sigma^u$ for smooth sections $f$ of
$\xox{0}{0}{v-u}$\;. We may compute how this operator changes under a change of
scale $\widehat\sigma=\Omega^{-1}\sigma$, for some nowhere vanishing
smooth function~$\Omega$. Firstly, note that
$\nabla_A(\Omega^{-u}f)=\Omega^{-u}(\nabla_A f-u\Upsilon_Af)$, where
$\Upsilon_A\equiv\Omega^{-1}\nabla_A\Omega$.
Hence,
\begin{align*}
\widehat\nabla_A(f\widehat\sigma^u)& \equiv (\nabla_A f)\widehat\sigma^u=
\Omega^{-u}(\nabla_A f)\sigma^u\\
& = \nabla_A(\Omega^{-u}f)\sigma^u+u\Upsilon_Af\widehat\sigma^u
=\nabla_A(\Omega^{-u}f\sigma^u)+u\Upsilon_Af\widehat\sigma^u
\end{align*}
and, writing $s=f\widehat\sigma^u\in\Gamma\big(\xox{u}{0}{v}\big)$, we obtain
\[\widehat\nabla_A s=\nabla_A s+u\Upsilon_As.\]
In particular, this transformation law records the invariance of $\nabla_A$
when $u=0$. Similarly, starting with
$\bar\nabla_A\colon \xox{-4}{0}{0}\to\xox{-4}{1}{-2}$ from
(\ref{array_of_operators}), in the presence of a scale $\sigma$, we obtain
\[\bar\nabla_A\colon \ \xox{u}{0}{v}\to\xox{0}{1}{-2}\otimes\xox{u}{0}{v},
\qquad\mbox{for all}\ u,v\in{\mathbb{R}}\]
and, under change of scale $\widehat\sigma=\Omega^{-1}\sigma$, we find
that
\[\widehat{\bar\nabla}_As=\bar\nabla_A s+v\bar\Upsilon_As,\]
where $\bar\Upsilon_A\equiv\Omega^{-1}\bar\nabla_A\Omega$.

Referring back to the Rumin complex (\ref{array_of_operators}) we also
have canonical differential operators (in the $E$ direction)
\begin{equation}\label{more_canonical_operators}
\xox{0}{1}{-2} \to \xox{-2}{2}{-2}\qquad\mbox{and}\qquad
\xox{-2}{1}{0} \to \xox{-3}{0}{1},
\end{equation}
which may be combined, via the Leibniz, rule with (\ref{densities}) to obtain,
in the presence of a scale, a~first-order differential operator
\[
\nabla_A\colon \ \xox{u}{1}{v}\to\ \xox{u-2}{2}{v}\;\oplus\ \xox{u-1}{0}{v+1}
\ =\;\xox{-2}{1}{0}\otimes\xox{u}{1}{v}.\]
To express the transformation of this operator under change of scale,
we may split it as
\[\phi_B\mapsto\nabla_A\phi_B=\nabla_{(A}\phi_{B)}+\nabla_{[A}\phi_{B]},\]
and recall that operator (\ref{densities}) on densities
$s\in\Gamma\big(\xox{u}{0}{v}\big)$ transforms as
\begin{equation}\label{transform_on_densities}
\widehat\nabla_A s=\nabla_A+u\Upsilon_A s,\qquad\mbox{where}\quad
\Upsilon_A\equiv\Omega^{-1}\nabla_A\Omega.\end{equation}

\begin{prop} Suppose we change scale $\sigma\in\Gamma\big(\xox{1}{0}{1}\big)$ by
$\widehat\sigma=\Omega^{-1}\sigma$. Then, for $\phi_B$ a~section of
$\xox{u}{1}{v}$ we have
\begin{equation}\label{full_transformation_law}\widehat\nabla_A\phi_B
=\nabla_A\phi_B+(u+1)\Upsilon_A\phi_B-\Upsilon_B\phi_A.\end{equation}
\end{prop}
\begin{proof}
It suffices to note that this transformation law is consistent with the
invariance of the operators~(\ref{more_canonical_operators}), which may be
written as
\[\phi_B\longmapsto\nabla_{(A}\phi_{B)}\qquad\mbox{and}\qquad
\phi_B\longmapsto\nabla_{[A}\phi_{B]}\]
and also with (\ref{transform_on_densities}) on densities.
\end{proof}

Similarly, from the canonical operators
\[\bar\nabla\colon \ \xox{-2}{1}{0}\to\xox{-2}{2}{-2}\qquad\mbox{and}\qquad
\bar\nabla\colon \ \xox{0}{1}{-2}\to\xox{1}{0}{-3},\]
from (\ref{array_of_operators}) we may construct, in the presence of a scale,
\[\Gamma\big(\xox{u}{1}{v}\big)\ni\phi_B\longmapsto\bar\nabla_A\phi_B\]
differentiating in the $F$ direction and transforming by
\begin{equation}\label{full_barred_transformation_law}
\widehat{\bar\nabla}_A\phi_B=\bar\nabla_A\phi_B
+(v+1)\bar\Upsilon_A\phi_B
-\bar\Upsilon_B\phi_A,
\qquad\mbox{where}\quad \bar\Upsilon_A\equiv\Omega^{-1}\bar\nabla_A\Omega.
\end{equation}

Finally, we may combine $\nabla_A$ and $\bar\nabla_A$ to define, in the
presence of a scale $\sigma\in\Gamma\big(\xox{1}{0}{1}\big)$, partial connections
\[\Gamma(\xox{u}{1}{v})\ni\phi_B\to
\left[\!\begin{array}{c}\nabla_A\phi_B\\[3pt]
\bar\nabla_A\phi_B\end{array}\!\right]
\in\begin{array}{c}\xox{-2}{1}{0}\\[-2pt] \oplus\\[-1pt]
\xox{0}{1}{-2}\end{array}\!\!\otimes\xox{u}{1}{v}
=\Wedge_H^1\otimes\xox{u}{1}{v}\]
and, indeed by the Leibniz rule, on all weighted spinor bundles. These partial
connections are generated by $\Wedge_H^1\to\Wedge_H^1\otimes\Wedge_H^1$ and
this basic one is characterised as follows.

\begin{prop}\label{prop6}
Let $\sigma$ be a nowhere vanishing section of $\xox{1}{0}{1}$,
equivalently a choice of contact form. Then, there is a unique partial
connection $\nabla_H\colon \Wedge_H^1\to\Wedge_H^1\otimes\Wedge_H^1$ so that the
induced partial connection on $\Wedge_H^4=\xox{-2}{0}{-2}$ annihilates
$\sigma^{-2}$ and so that~$\nabla_H$ has minimal partial torsion in the sense
that the induced operator $\Wedge_H^1\to\Wedge_{H\perp}^2$ agrees with the
Rumin operator~${\rm d}_\perp$ modulo the homomorphisms that are the obstructions to
integrability.
\end{prop}

\begin{proof} Recall that we constructed this partial connection from the Rumin
complex (\ref{array_of_operators}) modulo the obstructions to integrability.
The only ingredients in this argument not immediately visible in
${\rm d}_\perp\colon \wedge_H^1\to\Wedge_{H\perp}^2$, where the two operators
\[\xox{0}{0}{-4}\to\xox{-2}{1}{-4}\qquad\mbox{and}\qquad
\xox{-4}{0}{0}\to\xox{-4}{1}{-2}\]
coming from further along the Rumin complex. However, it is straightforward to
check that, for example,
\[\xox{0}{0}{-4}=\Wedge^5\otimes\big(\,\xox{-3}{0}{1}\big)^*\longrightarrow
\Wedge^5\otimes\big(\,\xox{-2}{1}{0}\big)^*=\xox{-2}{1}{-4}\]
is the adjoint of $\;\xox{-2}{1}{0}\to\xox{-3}{1}{0}$ so these two hidden
ingredients are also secretly carried by
${\rm d}_\perp\colon \Wedge_H^1\to\Wedge_{H\perp}^2$.
\end{proof}

In fact, this proposition also follows from the general theory \cite{CS}, or by
a more explicit spinor calculation~\cite{Moy}. The prolongation of
$\xox{0}{1}{0}\ni\phi_B\longmapsto
\big(\nabla_{(A}\phi_{B)},\bar\nabla_{(A}\phi_{B)}\big)$ gives an especially convenient
tractor bundle and its (partial) connection~(cf.~\cite{CS,Moy}).

A more familiar way \cite{OT} of saying that $\nabla_A$ annihilates the scale
$\sigma\in\Gamma\big(\xox{1}{0}{1}\big)$ and hence
$\sigma^{-1}\in\Gamma\big(\,\xox{-1}{0}{-1}\,\big)$, is to say that
$\nabla_A\epsilon_{BC}=0$ for the corresponding skew form $\epsilon_{BC}$
under~(\ref{epsilon}). As a final consistency check, we may verify that this
constraint is invariant under~(\ref{full_transformation_law}), as follows. For
$\phi_B$ a section of $\,\xox{-1}{1}{-1}$, the transformation~(\ref{full_transformation_law}) reads
\[\widehat\nabla_A\phi_B=\nabla_A\phi_B-\Upsilon_B\phi_A\]
(cf.~conformal spin geometry in four
dimensions~\cite[equation~(5.6.15)]{OT}) and, therefore, if $\phi_{BC}$ is a
section of $\,\xox{-1}{1}{-1}\,\otimes\,\xox{-1}{1}{-1}\,$, then the Leibniz
rule implies that
\[\widehat\nabla_A\phi_{BC}
=\nabla_A\phi_{BC}-\Upsilon_B\phi_{AC}-\Upsilon_C\phi_{BA}.\]
Hence, if $\phi_{BC}$ is skew, then
$\Upsilon_A\phi_{BC}=\Upsilon_B\phi_{AC}+\Upsilon_C\phi_{BA}$ and it follows
that
\[\widehat\nabla_A\phi_{BC}=\nabla_A\phi_{BC}-\Upsilon_A\phi_{BC}.\]
Finally, when $\widehat\sigma=\Omega^{-1}\sigma$, we find that
$\widehat{\epsilon}_{AB}=\Omega\epsilon_{AB}$ (cf.~\cite[equation~(5.6.2)]{OT})
and, hence, that
\[\widehat\nabla_A\widehat\epsilon_{BC}
=\widehat\nabla_A(\Omega\epsilon_{BC})
=\Omega\big(\widehat\nabla_A\epsilon_{BC}+\Upsilon_A\epsilon_{BC}\big)
=\Omega\nabla_A\epsilon_{BC},\]
as required.

\section{Flying saucers via spinors}\label{saucers}
Distilling the construction in \cite{EN2} down to its key ingredients, we
will explain how to construct a~$G_2$ contact structure starting from a
Legendrean contact structure plus some additional data, a~choice of
appropriately weighted sections of the Legendrean subbundles. We then will
calculate the torsion of the resulting~$G_2$ contact structure in terms of the
input data and the preferred partial connection in the previous section.

Using the Levi form we may identify $F^* = E \otimes L = E[-1,-1]$ and so
can write
\[\Wedge_H^1 \ni \omega_a = \begin{bmatrix} \omega_A \\ \bar{\omega}^A
\end{bmatrix} \in \begin{array}{c}
\hspace{2mm}E^*\\[-3pt] \oplus\\ E[-1,-1]\end{array}. \]
In what follows we will fix a scale $\sigma \in \xox{1}{0}{1}$, or equivalently
a contact form. The expression for the torsion will turn out to be independent
of this choice, but to write it down we will need the distinguished connection
from the previous section. The distinguished partial connection annihilating
the scale can be written on $E^*$ as
\[E^* \ni \omega_A \mapsto \begin{bmatrix} \nabla_A\omega_B \\
\bar{\nabla}^A\omega_B \end{bmatrix} \in \begin{array}{c}
\hspace{2mm}E^*\\[-3pt] \oplus\\ E[-1,-1] \end{array}
\hspace{-2mm} \otimes E^*.
\]
Recall from Section~\ref{CL} that a choice of scale gives rise to a skew spinor
$\epsilon_{AB}$, covariant constant under the distinguished partial connection,
together with its inverse $\epsilon^{AB}$ (such that
$\epsilon_{AB}\epsilon^{AC}=\delta_B{}^C$), which we may use to raise and lower
indices as per \cite{OT}. Covariant constancy implies that we may equate
$\nabla_A \sigma^B = \nabla_A\big(\sigma_C\epsilon^{BC}\big)$ and
$(\nabla_A\sigma_C)\epsilon^{BC}$, as in the usual spinor calculus.

We decompose $\Wedge_{H\perp}^2
= \Wedge^2E^* \oplus (E^* \otimes E)_\circ[-1,-1] \oplus \Wedge^2E[-2,-2]$,
and Proposition~\ref{prop6} then means we can write the Rumin operator as
\[d_\perp
\begin{bmatrix}
\sigma_A \\
\tau^A
\end{bmatrix} = \begin{bmatrix}
 -\nabla^A\sigma_A + \Pi_A\tau^A \\
\big(\nabla_A\tau^B - \bar{\nabla}^B\sigma_A\big)_\circ \\
 \bar{\nabla}_A\tau^A + \Sigma^A\sigma_A
\end{bmatrix},
\]
where $\Sigma^A$ and $\Pi_A$ are the obstructions to integrability and
$\big(\Xi_A{}^B\big)_\circ = \Xi_A{}^B - \frac{1}{2}\delta_A{}^B\Xi_C{}^C$
for $\Xi_A{}^B \in E \otimes E^*$.

Now suppose we are given sections
$o_A\in E^*[2,-1]$ and $\iota^A\in\ E[-2,1]$ such that
$o_A\iota^A = 1$,
classically known as a {\em spin-frame}~\cite[pp.~110--115]{OT}. Then we
may define a $G_2$ contact geometry as follows. These sections determine an
isomorphism
\[\Wedge^0[0,3] \oplus \Wedge^0[1,2] \oplus \Wedge^0[2,1] \oplus \Wedge^0[3,0]
\cong
E^*[2,2] \oplus E[1,1]\]
given by
\begin{equation}\textstyle\label{defining_isomorphism}
(x,y,z,w)\longmapsto
\big(xo_A-\frac{1}{\sqrt{3}}y\iota_A,w\iota^A-\frac{1}{\sqrt{3}}zo^A\big).
\end{equation}
If we set
$S = \Wedge^0[0,1] \oplus \Wedge^0[1,0]$,
we have an isomorphism
\[\textstyle\bigodot^3\!S\cong E^*[2,2] \oplus E[1,1] = \Wedge_H^1[2,2]\]
as required in the definition of a $G_2$ contact structure. The peculiar
factors here are chosen so that the isomorphism also satisfies the additional
Levi-compatibility condition.

Starting by considering the general connection on $S$ that
annihilates the scale, we can calculate the torsion of the $G_2$ contact
structure by calculating the obstruction to the differential operator
$\Lambda_H^1 \to \Lambda_{H\perp}^2$ induced by the defining isomorphism being
equal to the Rumin operator.

We can write any partial connection on $S$ as
\begin{gather*}
\nabla_{a} \begin{bmatrix}
y \\
z
\end{bmatrix}
= \begin{bmatrix}
\nabla_{a} y + \kappa_{a}y + \lambda_{a}z \\
\nabla_{a} z + \mu_{a} y + \nu_{a} z
\end{bmatrix}
\end{gather*}
for appropriately weighted sections $\kappa_{a}$, $\lambda_{a}$, $\mu_{a}$,
$\nu_{a}$, where we will take $\nabla_a$ to be the partial connection
distinguished by the contact form. The connection on $S$ will be compatible
with the contact form when $\kappa_{a} = -\nu_{a}$. Given this, the induced
partial connection on $\bigodot^3\!S \cong \Wedge^0[0,3] \oplus \Wedge^0[1,2]
\oplus \Wedge^0[2,1] \oplus \Wedge^0[3,0] $ is
\[\textstyle
\nabla_{a}
\begin{bmatrix}
x \\
y \\
z \\
w
\end{bmatrix} =
\begin{bmatrix}
\nabla_{a}x \\
\nabla_{a}y \\
\nabla_{a}z \\
\nabla_{a}w
\end{bmatrix} +
\begin{bmatrix}
3\kappa_{a} & \lambda_{a} & 0 & 0 \\
3\mu_{a}& \kappa_{a}& 2\lambda_{a}& 0 \\
0 & 2\mu_{a} & -\kappa_{a} & 3\lambda_{a}\\
0 & 0 & \mu_{a}& -3\kappa_{a}
\end{bmatrix}\begin{bmatrix}
x \\
y \\
z \\
w
\end{bmatrix} \in \Wedge_H^1 \otimes \bigodot^3\!S.
\]
Rewriting the right hand side using the defining isomorphism
(\ref{defining_isomorphism}) yields the differential operator
$\bigodot^3\!S \to \Wedge_H^1 \otimes (E^*[2,2] \oplus E[1,1])$ given by
\[\nabla_{a}
\begin{bmatrix}
x \\
y \\
z \\
w
\end{bmatrix} = \begin{bmatrix}
(\nabla_{a}x + 3x\kappa_{a} +
y\lambda_{a})o_A-\frac{1}{\sqrt{3}}(\nabla_{a}y +
3x\mu_{a} + y\kappa_{a} + 2z\lambda_{a})\iota_A \\[5pt]
-\frac{1}{\sqrt{3}}(\nabla_{a}z +
2y\mu_{a} - z\kappa_{a} + 3w\lambda_{a})o^A + (\nabla_{a}w
+ z\mu_{a} - 3w\kappa_{a})\iota^A
\end{bmatrix}.\]

To calculate the induced operator $\Wedge_H^1 \to \Wedge_{H\perp}^2$ we may use
the canonical identification \mbox{$\Wedge_H^1 = E^* \oplus E[-1,-1]$} and then
project the right hand side above onto the direct sum
\[\Wedge^2E^*[2,2] \oplus (E^* \otimes E)_\circ[1,1] \oplus \Wedge^2E.\]
We write $\kappa_a = \big(\kappa_A, \bar{\kappa}^A\big) \in E^* \oplus E[-1,1]$ and so
on to denote the projections. The induced operator (pulled back via the
isomorphism (\ref{defining_isomorphism})) is therefore
\[\textstyle\bigodot^3\!S[-2,-2] \to \Wedge^2E^*
\oplus (E^* \otimes E)_\circ[-1,-1] \oplus \Wedge^2E[-2,-2]\]
given by
\begin{gather*}
 \begin{bmatrix}
 x \\
 y \\
 z \\
 w
 \end{bmatrix} \mapsto  \begin{bmatrix}
\big(\nabla_{A}x+ 3x\kappa_{A}+ y\lambda_{A}\big)o^{A}
- \frac{1}{\sqrt{3}}\big(\nabla_{A}y + 3x\mu_{A}+ y\kappa_{A}
+ 2z\lambda_{A}\big)\iota^{A} \vspace{2mm}\\
\Big( {-}\frac{1}{\sqrt{3}}\big(\nabla_{A}z + 2y\mu_{A} - z\kappa_{A}
+ 3w\lambda_{A}\big)o^B + \big(\nabla_{A}w + z\mu_{A}
- 3w\kappa_{A}\big)\iota^B \vspace{2mm}\\
- \big(\bar{\nabla}^Bx + 3x\bar{\kappa}^B + y\bar{\lambda}^B\big)o_A
+ \frac{1}{\sqrt{3}}\big(\bar{\nabla}^By + 3x\bar{\mu}^B + y\bar{\kappa}^B
+ 2z\bar{\lambda}^B\big)\iota_A \Big)\rule[-5pt]{0pt}{5pt}_{\circ}\vspace{2mm}\\
\frac{1}{\sqrt{3}}\big(\bar{\nabla}^{A}z + 2y\bar{\mu}^{A} - z\bar{ \kappa}^{A}
+ 3w\bar{\lambda}^{A}\big)o_{A} - \big(\bar{\nabla}^{A}w + z\bar{\mu}^{A}
- 3w\bar{\kappa}^{A}\big)\iota_{A}
\end{bmatrix}.
\end{gather*}
This is the operator that we should compare to the Rumin operator (pulled back
via the isomorphism (\ref{defining_isomorphism})), which can be written
\[{\rm d}_\perp \begin{bmatrix} x \\ y \\z \\w \end{bmatrix} =
\begin{bmatrix}
-\nabla^A(xo_A) + \frac{1}{\sqrt{3}}\nabla^A(y\iota_A)
+\Pi_A\big(w\iota^A-\frac{1}{\sqrt{3}}zo^A\big) \vspace{2mm}\\
\big(\nabla_A\big(w\iota^B-\frac{1}{\sqrt{3}}zo^B\big)
- \bar{\nabla}^B\big(xo_A-\frac{1}{\sqrt{3}}y\iota_A\big)\big)_\circ \vspace{2mm}\\
\bar{\nabla}_A\big(w\iota^A\big) - \frac{1}{\sqrt{3}}\bar{\nabla}_A\big(zo^A\big)
+ \Sigma^A\big(xo_A-\frac{1}{\sqrt{3}}y\iota_A\big)
\end{bmatrix}.\]
Insisting that these differential operators are equal (and hence that the
torsion vanishes) we obtain a system of twelve spinor equations
\begin{gather*}
\nabla^Ao_A = 3\kappa^Ao_A - \sqrt{3}\mu^A\iota_A,
\qquad \nabla^A\iota_A = -\sqrt{3}\lambda^Ao_A + \kappa^A\iota_A, \\	
\Pi_Ao^A = -2\lambda^A\iota_A \hspace{67pt} \Pi_A\iota^A = 0, \\
 \big(\bar{\nabla}^Bo_A\big)_\circ
 = \big({-} \sqrt{3}\bar{\mu}^B\iota_A + 3\bar{\kappa}^Bo_A\big)_\circ,\\
\big(\bar{\nabla}^B\iota_A\big)_\circ
 = \big({-}2\mu_Ao^B - \sqrt{3}\bar{\lambda}^Bo_A + \bar{\kappa}^B\iota_A\big)_\circ, \\
\big(\nabla_Ao^B\big)_\circ
 = \big({-}\kappa_Ao^B - \sqrt{3}\mu_A\iota^B - 2\bar{\lambda}^B\iota_A\big)_\circ, \\
\big(\nabla_A\iota^B\big)_\circ
 = \big({-} \sqrt{3}\lambda_Ao^B - 3\kappa_A\iota^B \big)_\circ, \\
\bar{\nabla}_Ao^A
 = - \bar{\kappa}_Ao^A - \sqrt{3}\bar{\mu}_A\iota^A,\qquad
\bar{\nabla}_A\iota^A
 = -\sqrt{3}\bar{\lambda}_Ao^A - 3\bar{\kappa}_A\iota^A, \\
\Sigma^A\iota_A = 2\bar{\mu}_Ao^A, \qquad \Sigma^Ao_A = 0 .
\end{gather*}
Contracting the middle four equations above with combinations of
$o_A$, $\iota^B$ and their counterparts with indices raised and lowered,
respectively, produces (together with the other eight) a system of twenty
independent linear equations over $\mathbb{R}$ in twelve unknowns $\kappa^A
o_A, \mu^A\iota_A, \dots $ (owing to the trace-free condition, the middle four
equations above yield three independent equations each). This system
is consistent if and only if the following eight obstructions vanish:
\begin{gather*}
\begin{split}
&\psi_0 = \Pi_A\iota^A,\\
&\psi_1 = \textstyle
\Pi_Ao^A - \frac{2}{\sqrt{3}}\iota^A\big(\nabla_A\iota^B\big)\iota_B,\\
&\psi_2 = \iota_A\big(\bar{\nabla}^Ao_B\big)o^B + o_A\big(\bar{\nabla}^Ao_B\big)\iota^B
 - \bar{\nabla}_Ao^A + 2o_A\big(\bar{\nabla}^A\iota_B\big)o^B, \\
&\psi_3 = \textstyle\iota_A\big(\bar{\nabla}^A\iota_B\big)o^B
 + o_A\big(\bar{\nabla}^A\iota_B\big)\iota^B - \frac{1}{3}\bar{\nabla}_A\iota^A
 + \frac{2}{3}\iota_A\big(\bar{\nabla}^Ao_B\big)\iota^B
 + \frac{2}{\sqrt{3}}o^A\big(\nabla_Ao^B\big)o_B,\\
&\psi_4 = \textstyle
 o^A\big(\nabla_Ao^B\big)\iota_B + \iota^A\big(\nabla_Ao^B\big)o_B - \frac{1}{3}\nabla^Ao_A
 + \frac{2}{3}o^A\big(\nabla_A\iota^B\big)o_B
 + \frac{2}{\sqrt{3}}\iota_A\big(\bar{\nabla}^A\iota_B\big)\iota^B, \\
&\psi_5 =
 o^A\big(\nabla_A\iota^B\big)\iota_B + \iota^A\big(\nabla_A\iota^B\big)o_B
 - \nabla^A\iota_A + 2\iota^A\big(\nabla_Ao^B\big)\iota_B , \\
&\psi_6 = \textstyle \Sigma^A\iota_A + \frac{2}{\sqrt{3}}o_A\big(\bar{\nabla}^Ao_B\big)o^B,\\
&\psi_7 = \Sigma^Ao_A.
\end{split}
\end{gather*}
The above eight functions vanish with the invariant torsion of the $G_2$
contact structure. One can check using the formul{\ae}
(\ref{full_transformation_law}) and~(\ref{full_barred_transformation_law}) that
these expressions are invariant under change of scale, as they should be.

This construction and resulting formul{\ae} apply in some generality
(locally all~$G_2$ contact structures arise this way \cite{Moy}). In
particular, they generalise \cite[equation~(32)]{EN2} in case that the
spin-frame $o_A$, $\iota^A$ arises from flying saucer
data~\cite[equation~(24)]{EN2} (precisely, with the notation from~\cite{EN2},
this means that $\iota^A=\pi^!\psi$ and $o_A=\Theta^{-1}\pi^!\phi$).

\subsection*{Acknowledgements}

We would like to thank all staff at SIGMA in Kyiv for their extraordinary
courage, continuing their work despite the shocking Russian invasion and
unconscionable aggression.

We would also like to thank the referees for their careful reading of our
manuscript and for their valuable suggestions and corrections.

\pdfbookmark[1]{References}{ref}
\LastPageEnding


\begin{thebibliography}{99}
\footnotesize\itemsep=0pt

\bibitem{BE}
Baston R.J., Eastwood M.G., The {P}enrose transform: its interaction with
 representation theory, Dover Publications, Mineola, NY, 2016
 (Reprint of the 1989 edition, \textit{Oxford Mathematical Monographs}, The Clarendon Press,
 Oxford University Press, New York, 1989).

\bibitem{B}
Bryant R.L., Two exotic holonomies in dimension four, path geometries, and
 twistor theory, in Complex Geometry and {L}ie Theory ({S}undance, {UT},
 1989), \textit{Proc. Sympos. Pure Math.}, Vol.~53, \href{https://doi.org/10.1090/pspum/053/1141197}{Amer. Math. Soc.},
 Providence, RI, 1991, 33--88.

\bibitem{CSweyl}
\v{C}ap A., Slov\'ak J., Weyl structures for parabolic geometries,
 \href{https://doi.org/10.7146/math.scand.a-14413}{\textit{Math. Scand.}} \textbf{93} (2003), 53--90, \href{https://arxiv.org/abs/math.DG/0001166}{arXiv:math.DG/0001166}.

\bibitem{CS}
\v{C}ap A., Slov\'ak J., Parabolic geometries. {I}.~Background and general
 theory, \textit{Mathematical Surveys and Monographs}, Vol.~154, \href{https://doi.org/10.1090/surv/154}{Amer. Math.
 Soc.}, Providence, RI, 2009.

\bibitem{EG}
Eastwood M., Gover A.R., Prolongation on contact manifolds, \href{https://doi.org/10.1512/iumj.2011.60.4980}{\textit{Indiana
 Univ. Math.~J.}} \textbf{60} (2011), 1425--1486, \href{https://arxiv.org/abs/0910.5519}{arXiv:0910.5519}.

\bibitem{EN1}
Eastwood M., Nurowski P., Aerobatics of flying saucers, \href{https://doi.org/10.1007/s00220-019-03621-2}{\textit{Comm. Math.
 Phys.}} \textbf{375} (2020), 2335--2365, \href{https://arxiv.org/abs/1810.04852}{arXiv:1810.04852}.

\bibitem{EN2}
Eastwood M., Nurowski P., Aerodynamics of flying saucers, \href{https://doi.org/10.1007/s00220-019-03622-1}{\textit{Comm. Math.
 Phys.}} \textbf{375} (2020), 2367--2387, \href{https://arxiv.org/abs/1810.04855}{arXiv:1810.04855}.

\bibitem{E}
Engel F., Sur un groupe simple \'a quatorze param\`etres, \textit{C.~R.~Acad.
 Sci. Paris S\'er.~I Math.} \textbf{116} (1893), 786--788.

\bibitem{KMS}
Kol\'a\v{r} I., Michor P.W., Slov\'ak J., Natural operations in differential
 geometry, \href{https://doi.org/10.1007/978-3-662-02950-3}{Springer-Verlag}, Berlin, 1993.

\bibitem{LM}
Landsberg J.M., Manivel L., Legendrian varieties, \href{https://doi.org/10.4310/AJM.2007.v11.n3.a1}{\textit{Asian~J. Math.}}
 \textbf{11} (2007), 341--359, \href{https://arxiv.org/abs/math.AG/0407279}{arXiv:math.AG/0407279}.

\bibitem{Moy}
Moy T., Legendrean and $G_2$ contact structures, {M}aster {T}hesis, University
 of Adelaide, 2021, available at \url{https://hdl.handle.net/2440/132899}.

\bibitem{OT}
Penrose R., Rindler W., Spinors and space-time, {V}ol.~1, Two-spinor calculus
 and relativistic fields, \textit{Cambridge Monographs on Mathematical Physics},
 \href{https://doi.org/10.1017/CBO9780511564048}{Cambridge University Press}, Cambridge, 1984.

\bibitem{NT}
Penrose R., Rindler W., Spinors and space-time, {V}ol.~2, {S}pinor and twistor
 methods in space-time geometry, \textit{Cambridge Monographs on Mathematical Physics},
 \href{https://doi.org/10.1017/CBO9780511524486}{Cambridge University Press}, Cambridge, 1986.

\bibitem{R}
Rumin M., Un complexe de formes diff\'erentielles sur les vari\'et\'es de
 contact, \textit{C.~R.~Acad. Sci. Paris S\'er.~I Math.} \textbf{310} (1990),
 401--404.

\bibitem{S}
Sommers P., Space spinors, \href{https://doi.org/10.1063/1.524351}{\textit{J.~Math. Phys.}} \textbf{21} (1980),
 2567--2571.

\end{thebibliography}
\end{document}